\documentclass[a4paper,11pt]{article}
\usepackage{geometry}
\usepackage{microtype}
\usepackage{amssymb}
\usepackage{amsthm}
\usepackage{amsmath}
\usepackage{hyperref}
\newtheorem{theorem}{Theorem}[section]
\newtheorem{lemma}[theorem]{Lemma}
\newtheorem{corollary}[theorem]{Corollary}

\newtheorem{question}[theorem]{Question}
\theoremstyle{definition}
\newtheorem{example}[theorem]{Example}

\newcommand{\eps}{\varepsilon}
\newcommand{\Z}{\mathbb Z}
\newcommand{\kabel}[1]{\equiv_{#1}}
\newcommand{\growth}[1]{\mathcal P_{#1}}
\DeclareMathOperator{\sep}{sep}

\begin{document}

\title{Separating Many Words by Counting Occurrences of Factors}

\author{Aleksi Saarela\footnote{Department of Mathematics and Statistics, University of Turku, 20014 Turku, Finland, \texttt{amsaar@utu.fi}}}

\maketitle

\begin{abstract}
For a given language $L$, we study the languages $X$ such that for all distinct words $u, v \in L$, there exists a word $x \in X$ that appears a different number of times as a factor in $u$ and in $v$. In particular, we are interested in the following question: For which languages $L$ does there exist a finite language $X$ satisfying the above condition? We answer this question for all regular languages and for all sets of factors of infinite words.
\end{abstract}

\section{Introduction}

The motivation for this article comes from three sources.

First, a famous question about finite automata is the \emph{separating words problem}.
If $\sep(u, v)$ is the size of the smallest DFA
that accepts one of the words $u, v$ and rejects the other,
then what is the maximum of the numbers $\sep(u, v)$
when $u$ and $v$ run over all words of length at most $n$?
This question was first studied by Goral\v{c}\'{i}k and Koubek~\cite{goko86},
and they proved an upper bound $o(n)$ and a lower bound $\Omega(\log n)$.
The upper bound was improved to $O(n^{2 / 5} (\log n)^{3 / 5})$ by Robson~\cite{ro89},
and this remains the best known result.
A survey and some additional results can be found
in the article by Demaine, Eisentat, Shallit and Wilson~\cite{deeishwi11}.
Several variations of the problem exist.
For example,
NFAs~\cite{deeishwi11} or context-free grammars~\cite{cuperosh99} could be used instead of DFAs.
More generally, we could try to separate two disjoint languages $A$ and $B$
by providing a language $X$ from some specified family of languages such that
$A \subseteq X$ and $B \cap X = \varnothing$.
As an example related to logic, see~\cite{plze16}.
Alternatively, we could try to separate many words $w_1, \dots, w_k$
by providing languages $X_1, \dots, X_k$ with some specific properties such that
$w_i \in X_j$ if and only if $i = j$.
As an example, see~\cite{hoko11}.

Let $|w|_x$ denote the number of occurrences of a factor $x$ in a word $w$.
A simple observation that can be made about the separating words problem is that
if $|u|_x \ne |v|_x$, then $|u|_x \not\equiv |v|_x \pmod p$ for some relatively small prime $p$
(more specifically, $p = O(\log(|uv|))$),
and the number of occurrences modulo a prime can be easily counted by a DFA.
So if $u$ and $v$ have a different number of occurrences of some short factor $x$,
then $\sep(u, v)$ is small, see~\cite{deeishwi11} for more details.
Unfortunately, this approach does not provide any general bounds,
and more complicated ideas are required to prove the results mentioned in the previous paragraph.

In this article, we are interested in the question of
how well words can be separated
if we forget about automata and only consider the simple idea of counting occurrences of factors.
For any two distinct words $u$ and $v$ of length $n$,
we can find a factor $x$ of length $\lfloor n / 2 \rfloor + 1$ or less such that $|u|_x \ne |v|_x$.
A proof of this simple fact can be found in an article by Manuch~\cite{ma00}.
See~\cite{vygi14} for a variation
where also the positions of the occurrences modulo a certain number are taken into account.
The question becomes more interesting
if we want to separate more than two words (possibly infinitely many) at once,
and we can do this by counting the numbers of occurrences of more than one factor.
We are particularly interested in the following question.

\begin{question} \label{que:ssf}
    Given a language $L$, does there exist a finite language $X$ such that
    for all distinct words $u, v \in L$, there exists $x \in X$ such that $|u|_x \ne |v|_x$?
\end{question}

The second source of motivation is an old guessing game for two players,
let us call them Alice and Bob:
From a given set of options, Alice secretly picks one.
Bob is allowed to ask any yes-no questions,
and he is trying to figure out what Alice picked.
Two famous versions are the game ``Twenty Questions''
and the children's board game ``Guess Who''.
In their simplest forms, these kinds of games are easy to analyze:
The required number of questions
is logarithmic with respect to the number of options.
However, many more complicated variations have been studied.
As examples, see~\cite{pe87} and~\cite{amblsc02}.

In this article, we are interested in a variation
where the options are words and, instead of arbitrary yes-no questions,
Bob is allowed to ask for the number of occurrences of any factor in the word Alice has chosen.
Usually in games like this, Bob can decide every question based on the previous answers,
but we can also require that Bob needs to decide all the questions in advance.

\begin{question} \label{que:game}
    Given a language from which Alice has secretly picked one word $w$,
    can Bob find a finite language $X$ such that
    the answers to the questions ``What is $|w|_x$?'' for all $x \in X$
    are guaranteed to reveal the correct word $w$?
\end{question}

It is easy to see that Questions~\ref{que:ssf} and~\ref{que:game} are equivalent.
In this article, we will use the formulation of Question~\ref{que:ssf}
instead of talking about games.

The third source of motivation is $k$-abelian complexity.
For a positive integer $k$,
words $u$ and $v$ are said to be $k$-abelian equivalent
if $|u|_x = |v|_x$ for all factors $x$ of length at most $k$.
The factor complexity of an infinite word $w$
is a function that maps a number $n$ to the number of factors of $w$ of length $n$.
The $k$-abelian complexity of $w$ similarly
maps a number $n$ to the number of $k$-abelian equivalence classes of factors of $w$ of length $n$.
$k$-abelian equivalence was first studied by Karhum\"aki~\cite{ka80}.
Many basic properties were proved
by Karhum\"aki, Saarela and Zamboni in the article~\cite{kasaza13jcta},
where also $k$-abelian complexity was introduced.
Several articles have been published
about $k$-abelian complexity~\cite{cakasa17ejc,chluwu18,kasaza17esik},
and about abelian complexity (that is, the case $k = 1$) already earlier~\cite{risaza11}.
Perhaps the most interesting one from the point of view of this paper is~\cite{cakasa17ejc},
where the relationships between the $k$-abelian complexities of an infinite word
for different values of $k$ were studied.
However, the following simple question was not considered in that article.

\begin{question} \label{que:compl}
    Given an infinite word, does there exist a number $k \geq 1$ such that
    the $k$-abelian complexity of the word is the same as the usual factor complexity of the word?
\end{question}

For a given language, we can define its growth function and $k$-abelian growth function
as concepts analogous to the factor complexity and $k$-abelian complexity of an infinite word.
Then the above question can be generalized.
We are specifically interested in the case of regular languages.
Some connections between $k$-abelian equivalence and regular languages
have been studied by Cassaigne, Karhum\"aki, Puzynina and Whiteland~\cite{cakapuwh17}.

\begin{question} \label{que:growth}
    Given a language, does there exist a number $k \geq 1$ such that
    the growth function of the language
    is the same as the $k$-abelian growth function of the language?
\end{question}

In this article,
we first define some concepts related to Question~\ref{que:ssf}
and prove basic properties about them.
As stated above, Questions~\ref{que:ssf} and~\ref{que:game} are equivalent,
and so is Question~\ref{que:growth}, but this requires a short proof.
We answer these questions for two families of languages:
Sets of factors of infinite words (this corresponds to Question~\ref{que:compl})
and regular languages.
In the first case, the result is not surprising:
The answer is positive if and only if the word is ultimately periodic.
Our main result is a characterization in the case of regular languages:
The answer is positive if and only if the language does not have a subset of the form
$x w^* y w^* z$ for any words $w, x, y, z$ such that $wy \ne yw$.

\section{Preliminaries}

Throughout the article, we use the symbol $\Sigma$ to denote an alphabet.
All words are over $\Sigma$ unless otherwise specified.

\paragraph{Primitive words and Lyndon words.}
A nonempty word is \emph{primitive} if it is not a power of any shorter word.
The \emph{primitive root} of a nonempty word $w$
is the unique primitive word $p$ such that $w \in p^+$.
It is well known that nonempty words $u, v$ have the same primitive root
if and only if they commute, that is, $uv = vu$.

Words $u$ and $v$ are \emph{conjugates} if there exist words $p, q$ such that $u = pq$ and $v = qp$.
All conjugates of a primitive word are primitive.
If two nonempty words are conjugates, then their primitive roots are conjugates.

We can assume that the alphabet $\Sigma$ is ordered.
This order can be extended to a lexicographic order of $\Sigma^*$.
A \emph{Lyndon word} is a primitive word
that is lexicographically smaller than all of its other conjugates.
We use Lyndon words
when we need to pick a canonical representative from the conjugacy class of a primitive word.
The fact that this representative happens to be lexicographically minimal
is not actually important in this article.

The \emph{Lyndon root} of a nonempty word $w$
is the unique Lyndon word that is conjugate to the primitive root of $w$.
We state here the well-known periodicity theorem of Fine and Wilf~\cite{fiwi65},
and we use it to prove a simple result about Lyndon roots.

\begin{theorem}[Fine and Wilf] \label{thm:finewilf}
    Let $u, v$ be nonempty words.
    If the infinite words $u^\omega$ and $v^\omega$ have a common prefix of length
    \begin{math}
        |uv| - \gcd(|u|, |v|),
    \end{math}
    then $u$ and $v$ are powers of a common word of length $\gcd(|u|, |v|)$.
\end{theorem}

\begin{lemma} \label{lem:overlap-lyndon}
    Let $u, v$ be nonempty words.
    If $u^m$ and $v^n$ have a common factor of length $|uv|$,
    then $u$ and $v$ have the same Lyndon root.
\end{lemma}

\begin{proof}
    A factor of $u^m$ of length $|uv|$ is of the form $(u_1)^i u_2$,
    where $u_1$ is a conjugate of $u$, $u_2$ is a prefix of $u_1$, and $i \geq 1$.
    Similarly, a factor of $v^n$ of length $|uv|$ is of the form $(v_1)^j v_2$,
    where $v_1$ is a conjugate of $v$, $v_2$ is a prefix of $v_1$, and $j \geq 1$.
    If these factors are the same, then $(u_1)^i u_2 = (v_1)^j v_2$,
    so $u_1^\omega$ and $v_1^\omega$ have a common prefix of length $|uv|$.
    It follows from Theorem~\ref{thm:finewilf} that
    $u_1$ and $v_1$ are powers of a common word and therefore have the same primitive root.
    This primitive root is conjugate to the primitive roots of $u$ and $v$,
    so $u$ and $v$ have the same Lyndon root.
\end{proof}

\paragraph{Occurrences.}
Let $u$ and $w$ be words.
An \emph{occurrence of $u$ in $w$} is a triple $(x, u, y)$ such that $w = xuy$.
The number of occurrences of $u$ in $w$ is denoted by $|w|_u$.

Let $(x, u, y)$ and $(x', u', y')$ be occurrences in $w$.
If
\begin{equation*}
    \max(|x|, |x'|) < \min(|xu|, |x' u'|),
\end{equation*}
then we say that these occurrences have an \emph{overlap} of length
\begin{equation*}
    \min(|xu|, |x' u'|) - \max(|x|, |x'|).
\end{equation*}
If $|x| \geq |x'|$ and $|y| \geq |y'|$,
then we say that $(x, u, y)$ is \emph{contained} in $(x', u', y')$.

If $(x, u, y)$ is an occurrence in $w$ and $u \in L$,
then $(x, u, y)$ is an \emph{$L$-occurrence in $w$}.
It is a \emph{maximal $L$-occurrence in $w$}
if it is not contained in any other $L$-occurrence in $w$.

It is well known that
if $p$ is a primitive word, then $p$ cannot be a factor of $p^2$ in a nontrivial way,
or more formally,
$p^2$ does not have any other $p$-occurrences
than the trivial ones $(\eps, p, p)$ and $(p, p, \eps)$.
The following lemma is an easy consequence of this fact.

\begin{lemma} \label{lem:overlap-maxocc}
    Let $w$ be a word and $p$ be a primitive word.
    If two $p^+$-occurrences in $w$ have an overlap of length at least $|p|$,
    then they are contained in the same maximal $p^+$-occurrence.
    Moreover, every $p^+$-occurrence in $w$ is contained in exactly one maximal $p^+$-occurrence.
\end{lemma}

\begin{proof}
    To prove the first claim,
    let $(x, p^m, y)$ and $(x', p^n, y')$ be two $p^+$-occurrences in $w$ and let $|x| \leq |x'|$.
    If these occurrences have an overlap of length at least $|p|$,
    then the occurrence $(x', p, p^{n - 1} y')$ is contained in $(x, p^m, y)$.
    The number $|x'| - |x|$ must be divisible by $|p|$,
    because otherwise $p$ would be a factor of $p^2$ in a nontrivial way.
    Let $|x'| - |x| = kp$.
    Then $(x, p^{k + n}, y')$ is an occurrence in $w$.
    If $|y| \geq |y'|$,
    then both $(x, p^m, y)$ and $(x', p^n, y')$ are contained in $(x, p^{k + n}, y')$,
    which is contained in some maximal occurrence.
    On the other hand, if $|y| < |y'|$,
    then $(x', p^n, y')$ is contained in $(x, p^m, y)$,
    which is contained in some maximal occurrence.
    This proves the first claim.

    If a $p^+$-occurrence in $w$ is contained in two maximal $p^+$-occurrences,
    then those two maximal occurrences are contained in the same maximal occurrence
    by the first part of the proof.
    By the definition of maximality, these maximal occurrences are actually the same.
    This proves the second claim.
\end{proof}

\paragraph{$k$-abelian equivalence.}
Let $k$ be a positive integer.
Words $u, v \in \Sigma^*$ are \emph{$k$-abelian equivalent}
if $|u|_x = |v|_x$ for all $x \in \Sigma^{\leq k}$.
$k$-abelian equivalence is an equivalence relation
and it is denoted by $\kabel{k}$.

Here are some basic facts about $k$-abelian equivalence (see~\cite{kasaza13jcta}):
$u, v \in \Sigma^{\geq k - 1}$ are $k$-abelian equivalent if and only if
they have a common prefix of length $k - 1$ and $|u|_x = |v|_x$ for all $x \in \Sigma^k$.
The condition about prefixes can be replaced by a symmetric condition about suffixes.
Words of length $2k - 1$ or less are $k$-abelian equivalent if and only if they are equal.
$k$-abelian equivalence is a congruence, that is,
if $u \kabel{k} u'$ and $v \kabel{k} v'$, then $uv \kabel{k} u' v'$.

We are going to use the following simple fact a couple of times
when showing that two words are $k$-abelian equivalent:
If $u, v, w, x \in \Sigma^*$, $|v| = k - 1$, and $|x| = k$, then
\begin{equation*}
    |uvw|_x = |uv|_x + |vw|_x.
\end{equation*}

\begin{example}
    The words $aabab$ and $abaab$ are 2-abelian equivalent:
    They have the same prefix of length one,
    one occurrence of $aa$,
    two occurrences of $ab$,
    one occurrence of $ba$,
    and no occurrences of $bb$.

    The words $aba$ and $bab$ have the same number of occurrences of every factor of length two,
    but they are not 2-abelian equivalent,
    because they have a different number of occurrences of $a$.

    Let $k \geq 1$.
    The words $u = a^k b a^{k - 1}$ and $v = a^{k - 1} b  a^k$ are $k$-abelian equivalent:
    They have the same prefix of length $k - 1$,
    and $|u|_x = 1 = |v|_x$
    if $x = a^k$ or $x = a^i b a^{k - i - 1}$ for some $i \in \{0, \dots, k - 1\}$,
    and $|u|_x = 0 = |v|_x$ for all other factors $x$ of length $k$.
    On the other hand, $u$ and $v$ are not $(k + 1)$-abelian equivalent,
    because they have a different prefix of length $k$.
\end{example}

\paragraph{Growth functions and factor complexity.}
The \emph{growth function} of a language $L$ is the function
\begin{equation*}
    \growth{L}: \Z_{\geq 0} \to \Z_{\geq 0},\ \growth{L}(n) = |L \cap \Sigma^n|
\end{equation*}
mapping a number $n$ to the number of words of length $n$ in $L$.
The \emph{factor complexity} of an infinite word $w$, denoted by $\growth{w}$,
is the growth function of the set of factors of $w$
(technically, the domain of $\growth{w}$ is often defined to be $\Z_+$ instead of $\Z_{\geq 0}$).

We can also define $k$-abelian versions of these functions.
The \emph{$k$-abelian growth function} of a language $L$ is the function
\begin{equation*}
    \growth{L}^k: \Z_{\geq 0} \to \Z_{\geq 0},\ \growth{L}(n) = |(L \cap \Sigma^n) / \kabel{k}|,
\end{equation*}
where $(L \cap \Sigma^n) / \kabel{k}$
denotes the set of equivalence classes of elements of $L \cap \Sigma^n$.
The \emph{$k$-abelian complexity} of an infinite word $w$, denoted by $\growth{w}^k$,
is the $k$-abelian growth function of the set of factors of $w$.

An infinite word $w$ is \emph{ultimately periodic}
if there exist finite words $u, v$ such that $w = uv^\omega$.
An infinite word is \emph{aperiodic} if it is not ultimately periodic.
It was proved by Morse and Hedlund~\cite{mohe38}
that if $w$ is ultimately periodic, then $\growth{w}(n) = O(1)$,
and if $w$ is aperiodic, then $\growth{w}(n) \geq n + 1$ for all $n$.

\section{Separating sets of factors}

A language $X$ is a \emph{separating set of factors} (SSF) of a language $L$
if for all distinct words $u, v \in L$, there exists $x \in X$ such that
\begin{math}
    |u|_x \ne |v|_x.
\end{math}
The set $X$ is \emph{size-minimal}
if no set of smaller cardinality is an SSF of $L$,
and it is \emph{inclusion-minimal}
if $X$ does not have a proper subset that is an SSF of $L$.

\begin{example}
    Let $\Sigma = \{a, b\}$.
    The language $a^*$
    has two inclusion-minimal SSFs:
    $\{\eps\}$ and $\{a\}$.
    Both of them are also size-minimal.
    The language $\Sigma^2 = \{aa, ab, ba, bb\}$
    has eight inclusion-minimal SSFs:
    \begin{equation*}
        \{a, ab\}, \{a, ba\}, \{b, ab\}, \{b, ba\},
        \{aa, ab, ba\}, \{aa, ab, bb\}, \{aa, ba, bb\}, \{ab, ba, bb\}.
    \end{equation*}
    The first four are size-minimal.
\end{example}

\begin{example}
    Let $\Sigma = \{a, b, c, d, e, f\}$.
    The language $L = \{ac, ad, be, bf\}$ has a size-minimal SSF $\{a, c, e\}$.
    In terms of the guessing game mentioned in the introduction,
    this means that if Alice has chosen $w \in L$,
    then Bob can ask for the numbers $|w|_a, |w|_c, |w|_e$, and this will always reveal $w$.
    Actually, two questions are enough
    if Bob can choose the second question after hearing the answer to the first one:
    He can first ask for $|w|_a$,
    and then for either $|w|_c$ or $|w|_e$ depending on whether $|w|_a = 1$ or $|w|_a = 0$.
\end{example}

The following lemma contains some very basic results related to the above defitions.
In particular, it proves that every language has an inclusion-minimal SSF,
and all SSFs are completely characterized by the inclusion-minimal ones.

\begin{lemma} \label{lem:minimal}
    Let $L$ and $X$ be languages.
    \begin{enumerate}
        \item
        If $L \ne \varnothing$,
        then $L$ has a proper subset that is an SSF of $L$.
        \item
        If $X$ is an SSF of $L$ and $K \subseteq L$, then $X$ is an SSF of $K$.
        \item
        If $X$ is an SSF of $L$ and $X \subseteq Y$, then $Y$ is an SSF of $L$.
        \item
        If $X$ is an SSF of $L$,
        then $X$ has a subset that is an inclusion-minimal SSF of $L$.
    \end{enumerate}
\end{lemma}

\begin{proof}
    To prove the first claim,
    let $w \in L$ be of minimal length and let $X = L \smallsetminus \{w\}$.
    Let $u, v \in L$ and $u \ne v$.
    By symmetry, we can assume that $|u| \leq |v|$ and $v \ne w$.
    Then $v \in X$ and $|u|_v = 0 \ne 1 = |v|_v$.
    This shows that $X$ is an SSF of $L$.

    The second and third claims follow directly from the definition of an SSF.

    The fourth claim is easy to prove if $X$ is finite.
    In the general case, it can be proved by Zorn's lemma as follows.
    Consider the partially ordered (by inclusion) family of all subsets of $X$ that are SSFs of $L$.
    The family contains at least $X$, so it is nonempty.
    By Zorn's lemma, if every nonempty chain
    (that is, a totally ordered subset of the family)
    $C$ has a lower bound in this family,
    then the family has a minimal element, which is then an inclusion-minimal SSF of $L$.
    We show that the intersection $I$ of the sets in $C$ is an SSF of $L$,
    and therefore it is the required lower bound.
    For any $u, v \in L$ such that $u \ne v$ and for any $Y \in C$,
    there exists $y \in Y$ such that $|u|_y \ne |v|_y$.
    Then $y$ must be a factor of $u$ or $v$,
    so if $u$ and $v$ are fixed, then there are only finitely many possibilities for $y$.
    Thus at least one of the words $y$ is in all sets $Y$ and therefore also in $I$.
    This shows that $I$ is an SSF of $L$.
    This completes the proof.
\end{proof}

The next lemma shows a connection between SSFs and $k$-abelian equivalence.

\begin{lemma} \label{lem:kabel}
    Let $L$ be a language.
    \begin{enumerate}
        \item
        Let $k \in \Z_+$.
        The language $\Sigma^{\leq k}$ is an SSF of $L$
        if and only if the words in $L$ are pairwise $k$-abelian nonequivalent.
        \item
        The language $L$ has a finite SSF
        if and only if there exists a number $k$ such that
        the words in $L$ are pairwise $k$-abelian nonequivalent.
    \end{enumerate}
\end{lemma}

\begin{proof}
    The first claim follows directly from the definitions of an SSF and $k$-abelian equivalence.
    The ``only if'' and ``if'' directions of the second claim can be proved as follows:
    If a finite set $X$ is an SSF of $L$,
    then $X \subseteq \Sigma^{\leq k}$ for some $k$,
    and then the words in $L$ are pairwise $k$-abelian nonequivalent.
    Conversely, if the words in $L$ are pairwise $k$-abelian nonequivalent,
    then $\Sigma^{\leq k}$ is an SSF of $L$.
\end{proof}

Note that the condition ``the words in $L$ are pairwise $k$-abelian nonequivalent''
can be equivalently expressed as ``$\growth{L} = \growth{L}^k$''.
This means that Lemma~\ref{lem:kabel}
proves the equivalence of Questions~\ref{que:ssf} and~\ref{que:growth}.

\begin{example}
    Let $w, x, y, z \in \{a, b\}^*$ and $L = \{awa, axb, bya, bzb\}$.
    No two words in $L$ have both a common prefix and a common suffix of length one,
    so the words are pairwise 2-abelian nonequivalent.
    By the first claim of Lemma~\ref{lem:kabel}, $\{a, b\}^{\leq 2}$ is an SSF of $L$.
    This SSF is not size-minimal
    (by the first claim of Lemma~\ref{lem:minimal}, $L$ has an SSF of size three),
    but it has the advantage of consisting of very short words and not depending on $w, x, y, z$.
    Actually, also $\{\eps, a, aa, ab, ba\}$ is an SSF of $L$.
    This follows from the fact that
    $|u|_b = |u|_\eps - |u|_a - 1$ and $|u|_{bb} = |u|_\eps - |u|_{aa} - |u|_{ab} - |u|_{ba} - 2$
    for all $u \in \{a, b\}^*$.
\end{example}

\begin{example}
    In a list of about 140000 English words
    (found in the SCOWL database~\footnote{http://wordlist.aspell.net/}),
    there are no 4-abelian equivalent words.
    Therefore, by Lemma~\ref{lem:kabel},
    $\Sigma^{\leq 4}$ is an SSF of the language formed by these words
    (the alphabet $\Sigma$ here contains the 26 letters from $a$ to $z$
    and also many accented letters and other symbols).
    The only pairs of 3-abelian equivalent words are
    $reregister, registerer$ and $reregisters, registerers$.
    The number of other pairs of 2-abelian equivalent words is also small enough
    that they can be listed here:
    \begin{align*}
        &indenter, intender&
        &indenters, intenders\\
        &pathophysiologic, physiopathologic&
        &pathophysiological, physiopathological\\
        &pathophysiology, physiopathology&
        &pathophysiologies, physiopathologies\\
        &tamara, tarama&
        &tamaras, taramas\\
        &tantarara, tarantara&
        &tantararas, tarantaras\\
        &tantaras, tarantas
    \end{align*}
    This means that most words of length 4 and 3 are not needed in the SSF.
    For example,
    the set $\Sigma^{\leq 2} \cup \{rere, hop, ind, tan, tar\}$ is an SSF of the language.
    We did not try to find a minimal SSF.
\end{example}

In the next lemma,
we consider whether the properties of having or not having a finite SSF
are preserved under the rational operations union, concatenation and Kleene star.

\begin{lemma}
    Let $K$ and $L$ be languages.
    \begin{enumerate}
        \item
        If $L$ has a finite SSF and $F$ is a finite language,
        then $L \cup F$ has a finite SSF.
        \item
        If $L$ does not have a finite SSF,
        then $L \cup K$ does not have a finite SSF.
        \item
        If $L$ has a finite SSF and $w$ is a word,
        then $wL$ and $Lw$ have finite SSFs.
        \item
        If $L$ does not have a finite SSF and $K \ne \varnothing$,
        then neither $KL$ nor $LK$ have finite SSFs.
        \item
        $L^*$ has a finite SSF
        if and only if
        there exists a word $w$ such that $L \subseteq w^*$.
        \item
        If the symmetric difference of $K$ and $L$ is finite,
        then either both or neither have a finite SSF.
    \end{enumerate}
\end{lemma}

\begin{proof}
    \begin{enumerate}
        \item
        Let $X$ be a finite SSF of $L$.
        Let $u, v \in L \cup F$ and $u \ne v$.
        First, if $u, v \in L$, then $|u|_x \ne |v|_x$ for some $x \in X$.
        Second, if $u \in F$ and $|u| = |v|$, then $|u|_u \ne |v|_u$.
        Finally, if $|u| \ne |v|$, then $|u|_\eps \ne |v|_\eps$.
        Thus $X \cup F \cup \{\eps\}$ is an SSF of $L \cup F$.
        \item
        If a finite set is an SSF of $L \cup K$, then it is also an SSF of $L$.
        \item
        Let $wL$ have no finite SSF.
        Let $k \in \Z_+$ and $k' = k + |w|$.
        By Lemma~\ref{lem:kabel},
        there exist two $k'$-abelian equivalent words $wu, wv \in wL$.
        Then $u$ and $v$ have a common prefix $p$ of length $k - 1$.
        For all $x \in \Sigma^k$,
        \begin{equation*}
            |u|_x = |wu|_x - |wp|_x = |wv|_x - |wp|_x = |v|_x,
        \end{equation*}
        so $u \kabel{k} v$.
        We have shown that for all $k \geq 1$,
        there exist two $k$-abelian equivalent words in $L$.
        By Lemma~\ref{lem:kabel}, $L$ does not have a finite SSF.
        The case of $Lw$ is symmetric.
        \item
        Let $L$ have no finite SSF and let $w \in K$.
        Let $k \in \Z_+$.
        By Lemma~\ref{lem:kabel},
        there exist two $k$-abelian equivalent words $u, v \in L$,
        and then $wu, wv \in KL$ are $k$-abelian equivalent.
        We have shown that for all $k \geq 1$,
        there exist two $k$-abelian equivalent words in $KL$.
        By Lemma~\ref{lem:kabel}, $KL$ does not have a finite SSF.
        The case of $LK$ is symmetric.
        \item
        If $L \subseteq w^*$, then $\{w\}$ is an SSF of $L$.
        If there does not exist $w$ such that $L \subseteq w^*$,
        then there exist $u, v \in L$ such that $uv \ne vu$.
        For all $k \in \Z_+$, the words $u^k v u^{k - 1}, u^{k - 1} v u^k \in L^*$ are distinct.
        They have the same prefix of length $k - 1$.
        If $u_1$ is the prefix and $u_2$ is the suffix of $u^{k - 1}$ of length $k - 1$, then
        \begin{equation*}
            |u^k v u^{k - 1}|_x
            = |u^k|_x + |u_2 v u_1|_x + |u^{k - 1}|_x
            = |u^{k - 1}|_x + |u_2 v u_1|_x + |u^k|_x
            = |u^{k - 1} v u^k|_x
        \end{equation*}
        for all $x \in \Sigma^k$, so $u^k v u^{k - 1} \kabel{k} u^{k - 1} v u^k$.
        We have shown that for all $k \geq 1$,
        there exist two $k$-abelian equivalent words in $L^*$.
        By Lemma~\ref{lem:kabel}, $L^*$ does not have a finite SSF.
        \item
        If $K$ has a finite SSF, then so does $K \cap L$.
        If $L \smallsetminus K$ is finite,
        then also $L$ has a finite SSF by the first claim of this lemma.
        Similarly, if $L$ has a finite SSF and $K \smallsetminus L$ is finite,
        then also $K$ has a finite SSF.\qedhere
    \end{enumerate}
\end{proof}

\begin{example}
    We give an example showing that
    the property of having a finite SSF is not always preserved by union and concatenation.
    Let
    \begin{math}
        L = \{a^k b a^{k - 1} \mid k \in \Z_+\}.
    \end{math}
    Then both $L$ and $Laa$ have the finite SSF $\{\eps\}$.
    On the other hand,
    $L \{\eps, aa\} = L \cup Laa$ contains the $k$-abelian equivalent words
    $a^k b a^{k - 1}$ and $a^{k - 1} b a^k$ for all $k \geq 2$,
    so by Lemma~\ref{lem:kabel},
    $L \cup Laa$ does not have a finite SSF even though both $L$ and $Laa$ do have a finite SSF,
    and $L \{\eps, aa\}$ does not have a finite SSF
    even though both $L$ and $\{\eps, aa\}$ do have a finite SSF.
\end{example}

\section{Infinite words}

In this section, we give an answer to Question~\ref{que:compl}.

\begin{theorem} \label{thm:infinitewords}
    Let $w$ be an infinite word.
    There exists $k \in \Z_+$ such that $\growth{w} = \growth{w}^k$
    if and only if
    $w$ is ultimately periodic.
\end{theorem}

\begin{proof}
    First, let $w$ be ultimately periodic.
    Then we can write $w = u v^\omega$,
    where $v$ is primitive and $v$ is not a suffix of $u$.
    Let $k = |uv| + 1$ and let $x, y$ be $k$-abelian equivalent factors of $w$.
    If $x$ and $y$ are shorter than $uv$, then $x = y$.
    Otherwise $x$ and $y$ have a common prefix of length $k - 1 = |uv|$
    and we can write $x = u' v' x'$ and $y = u' v' y'$, where $|u'| = |u|$ and $|v'| = |v|$.
    Here $v'$ is a factor of $v^\omega$, so it must be a conjugate of $v$,
    and it is followed by a $(v')^\omega$.
    Thus $x'$ and $y'$ are prefixes of $(v')^\omega$ and they are of the same length,
    so $x' = y'$ and thus $x = y$.
    We have proved that no two factors of $w$ are $k$-abelian equivalent.
    It follows that $\growth{w} = \growth{w}^k$.

    Second, let $w$ be aperiodic and let $k \geq 2$ be arbitrary.
    Let $n = \growth{w}(k - 1) + 1$.
    There must exist a word $u$ of length $(k - 1) n$
    that occurs infinitely many times in $w$ as a factor.
    We can write $u = x_1 \dotsm x_n$, where $x_1, \dots, x_n \in \Sigma^{k - 1}$.
    By the definition of $n$,
    there exist two indices $i, j \in \{1, \dots, n\}$ such that $x_i = x_j$.
    Let $i < j$, $x = x_i = x_j$ and $y = x_{i + 1} \dotsm x_{j - 1}$.
    Then $xyx$ is a factor of $u$ and thus occurs infinitely many times in $w$ as a factor.
    Therefore we can write $w = z_0 xyx z_1 xyx z_2 xyx \dotsm$
    for some infinite sequence of words $z_0, z_1, z_2, \dotsc$.
    If the words $xy$ and $x z_i$ have the same primitive root $p$ for all $i \in \Z_+$,
    then $w = z_0 p^\omega$, which contradicts the aperiodicity of $w$.
    Thus there exists $i$ such that $xy$ and $x z_i$ have a different primitive root.
    Then $xy x z_i \ne x z_i xy$ and thus $xyx z_i x \ne x z_i xyx$.
    On the other hand, $xyx z_i x$ and $x z_i xyx$ are $k$-abelian equivalent
    because they have the same prefix $x$ of length $k - 1$ and
    \begin{equation*}
        |xyx z_i x|_t = |xyx|_t + |x z_i x|_t = |x z_i x|_t + |xyx|_t = |x z_i xyx|_t
    \end{equation*}
    for all $t \in \Sigma^k$.
    Moreover, $xyx z_i x$ and $x z_i xyx$ are factors of $w$.
    It follows that $\growth{w} \ne \growth{w}^k$.
\end{proof}

\begin{corollary}
    The set of factors of an infinite word $w$ has a finite SSF
    if and only if
    $w$ is ultimately periodic.
\end{corollary}

\begin{proof}
    Follows from Theorem~\ref{thm:infinitewords} and Lemma~\ref{lem:kabel}.
\end{proof}

\section{Regular languages}

In this section, we give an answer to Question~\ref{que:ssf} for regular languages.

\begin{lemma} \label{lem:half-main}
    If a language $L$ has a subset of the form $x w^* y w^* z$
    for some words $w, x, y, z$ such that $wy \ne yw$,
    then $L$ does not have a finite SSF.
\end{lemma}

\begin{proof}
    For all $k \in \Z_+$, the words $x w^k y w^{k - 1} z$ and $x w^{k - 1} y w^k z$ are distinct.
    They have the same prefix of length $k - 1$.
    If $w_1$ is the prefix and $w_2$ is the suffix of $w^{k - 1}$ of length $k - 1$, then
    \begin{equation*}
        |x w^k y w^{k - 1} z|_t
        = |x w_1|_t + |w^k|_t + |w_2 y w_1|_t + |w^{k - 1}|_t + |w_2 z|_t
        = |x w^{k - 1} y w^k z|_t
    \end{equation*}
    for all $t \in \Sigma^k$, so $x w^k y w^{k - 1} z \kabel{k} x w^{k - 1} y w^k z$.
    We have shown that for all $k \geq 1$,
    there exist two $k$-abelian equivalent words in $L$.
    By Lemma~\ref{lem:kabel}, $L$ does not have a finite SSF.
\end{proof}

A language $L$ is \emph{bounded} if it is a subset of a language of the form
\begin{equation*}
    v_1^* \dotsm v_n^*,
\end{equation*}
where $v_1, \dots, v_n$ are words.
It was proved by Ginsburg and Spanier~\cite{gisp66} that
a regular language is bounded if and only if it is a finite union of languages of the form
\begin{equation*}
    u_0 v_1^* u_1 \dotsm v_n^* u_n,
\end{equation*}
where $u_0, \dots, u_n$ are words and $v_1, \dots, v_n$ are nonempty words.

\begin{lemma} \label{lem:unbounded}
    Every regular language is bounded
    or has a subset of the form $x w^* y w^* z$
    for some words $w, x, y, z$ such that $wy \ne yw$.
\end{lemma}

\begin{proof}
    The proof is by induction.
    Every finite language is bounded.
    We assume that $A$ and $B$ are regular languages that have the claimed property
    and prove that also $A \cup B$, $AB$ and $A^*$ have the claimed property.

    First, we consider $A \cup B$.
    If both $A$ and $B$ are bounded,
    then so is $A \cup B$ by the characterization of Ginsburg and Spanier.
    If at least one of $A$ and $B$ has a subset of the form $x w^* y w^* z$
    for some words $w, x, y, z$ such that $wy \ne yw$,
    then $A \cup B$ has this same subset.

    Next, we consider $AB$.
    If both $A$ and $B$ are bounded or if one of them is empty,
    then $AB$ is bounded by the definition of bounded languages.
    If $A$ and $B$ are nonempty and at least one of them has a subset of the form $x w^* y w^* z$
    for some words $w, x, y, z$ such that $wy \ne yw$,
    then $AB$ has a subset of the same form with a different $x$ or $z$.

    Finally, we consider $A^*$.
    If $A \subseteq u^*$ for some word $u$, then $A^* \subseteq u^*$ is bounded.
    If $A$ is not a subset of $u^*$ for any word $u$,
    then there exist $w, y \in A$ such that $wy \ne yw$, and $A^*$ has $w^* y w^*$ as a subset.
\end{proof}

By Lemmas~\ref{lem:half-main} and~\ref{lem:unbounded},
if a regular language is not bounded, then it does not have a finite SSF.
Thus we can concentrate on bounded regular languages.
We continue with a technical lemma.

\begin{lemma} \label{lem:nkp}
    Let $L$ be a bounded regular language.
    There exist numbers $n, k \geq 0$ and a finite set of Lyndon words $P$
    such that the following are satisfied:
    \begin{enumerate}
        \item \label{cond:1}
        If $p, q \in P$, $p \ne q$, and $l, m \geq 0$,
        then $p^l$ and $q^m$ do not have a common factor of length $n$.
        \item \label{cond:2}
        If $u \in L$ and $p \in P$,
        then either there is at most one maximal $p^{\geq n}$-occurrence in $u$
        or $L$ has a subset of the form $x (p^m)^* y (p^m)^* z$,
        where $py \ne yp$ and $m \geq 1$.
        \item \label{cond:3}
        If $u \in L$ and $x$ is a factor of $u$ of length at least $k$,
        then $x$ has a factor $p^{n + 1}$ for some $p \in P$.
    \end{enumerate}
\end{lemma}

\begin{proof}
    If $L$ is finite, then the claim is basically trivial.
    For example, we can let $P = \varnothing$ and $n = k = \max \{|w| \mid w \in L\} + 1$.
    If $L$ is infinite, then we can write
    \begin{equation*}
        L = \bigcup_{i = 1}^s u_{i 0} \prod_{j = 1}^{r_i} v_{i j}^* u_{i j},
    \end{equation*}
    where $s \geq 1$ and $r_1, \dots, r_s \geq 0$ are numbers,
    $r_i \geq 1$ for at least one $i$,
    all the $u_{i j}$ are words,
    and all the $v_{i j}$ are nonempty words.
    We are going to prove that the three conditions are satisfied if
    $P$ is the set of Lyndon roots of the words $v_{i j}$ and
    \begin{align*}
        n &= 2 \cdot \max \Big\{
            |u_{i 0} \prod_{j = 1}^{r_i} v_{i j} u_{i j}|
            \ \Big|\ i \in \{1, \dots, s\}
        \Big\},\\
        k &= \max \Big\{
            |u_{i 0} \prod_{j = 1}^{r_i} v_{i j}^{n + 2} u_{i j}|
            \ \Big|\ i \in \{1, \dots, s\}
        \Big\}.
    \end{align*}

    First, we prove Condition~\ref{cond:1}.
    If $p, q \in P$, $l, m \geq 0$, and $p^l$ and $q^m$ have a common factor of length $|pq|$,
    then $p = q$ by Lemma~\ref{lem:overlap-lyndon}.
    Clearly $n \geq |pq|$, so Condition~\ref{cond:1} is satisfied.

    Next, we prove Condition~\ref{cond:2}.
    This is the most complicated part of the proof.
    Let $u \in L$ and $p \in P$.
    There are numbers $i, m_1, \dots, m_{r_i}$ such that
    \begin{equation*}
        u = u_{i 0} \prod_{j = 1}^{r_i} v_{i j}^{m_j} u_{i j}.
    \end{equation*}
    Let $(w_1, p^N, w_2)$ be a maximal $p^{\geq n}$-occurrence in $u$.
    If there does not exist an index $J$ such that $(w_1, p^N, w_2)$ and the occurrence
    \begin{equation} \label{eq:vijocc}
        \Big( u_{i 0} \prod_{j = 1}^{J - 1} v_{i j}^{m_j} u_{i j},\
        v_{i J}^{m_J},\
        u_{i J} \prod_{j = J + 1}^{r_i} v_{i j}^{m_j} u_{i j} \Big)
    \end{equation}
    have an overlap of length at least $|p v_{i J}|$, then
    \begin{equation*}
        |p^N| < \sum_{j = 0}^{r_i} |u_{i j}| + \sum_{j = 1}^{r_i} |p v_{i j}|
        \leq \frac{n}{2} + r_i |p|
        \leq \frac{n}{2} + \frac{n}{2} \cdot |p|
        \leq n |p|,
    \end{equation*}
    which is a contradiction.
    So there exists a number $J$ such that $(w_1, p^N, w_2)$ and \eqref{eq:vijocc}
    have an overlap of length at least $|p v_{i J}|$,
    and then $p$ is the Lyndon root of $v_{i J}$ by Lemma~\ref{lem:overlap-lyndon}.
    We can write $v_{i J}^{m_J} = p_1 p^M p_2$,
    where $p_1$ is a proper suffix of $p$, $p_2$ is a proper prefix of $p$, and $M \geq 1$.
    Then the occurrences $(w_1, p^N, w_2)$ and
    \begin{equation} \label{eq:pmocc}
        \Big( u_{i 0} \prod_{j = 1}^{J - 1} v_{i j}^{m_j} u_{i j} \cdot p_1,\
        p^M,\
        p_2 u_{i J} \prod_{j = J + 1}^{r_i} v_{i j}^{m_j} u_{i j} \Big)
    \end{equation}
    have an overlap of length at least $|p|$,
    so \eqref{eq:pmocc} is contained in $(w_1, p^N, w_2)$ by Lemma~\ref{lem:overlap-maxocc}.
    If there is another maximal $p^{\geq n}$-occurrence $(w_1', p^{N'}, w_2')$ in $u$,
    then similarly there exists a number $J'$ such that $p$ is the Lyndon root of $v_{i J'}$,
    $v_{i J'}^{m_J'} = p_1' p^{M'} p_2'$,
    where $p_1'$ is a proper suffix of $p$, $p_2'$ is a proper prefix of $p$, and $M' \geq 1$,
    and the occurrence
    \begin{equation} \label{eq:pmpocc}
        \Big( u_{i 0} \prod_{j = 1}^{J' - 1} v_{i j}^{m_j} u_{i j} \cdot p_1',\
        p^{M'},\
        p_2' u_{i J'} \prod_{j = J' + 1}^{r_i} v_{i j}^{m_j} u_{i j} \Big)
    \end{equation}
    is contained in the occurrence $(w_1', p^{N'}, w_2')$.
    It must be $J \ne J'$,
    because otherwise \eqref{eq:pmocc} and \eqref{eq:pmpocc} would be the same,
    and then the maximal occurrences $(w_1, p^N, w_2)$ and $(w_1', p^{N'}, w_2')$ would be the same
    by Lemma~\ref{lem:overlap-maxocc}.
    By symmetry, we can assume $J < J'$.
    Then $L$ has a subset of the form $x (p^{l_1})^* y (p^{l_2})^* z$, where
    \begin{equation*}
        y = p_2 u_{i J} \prod_{j = J + 1}^{J' - 1} v_{i j}^{m_j} u_{i j} \cdot p_1',
    \end{equation*}
    and then it also has the subset $x (p^m)^* y (p^m)^* z$, where $m = l_1 l_2$.
    Here $y \notin p^*$ and thus $py \ne yp$,
    because otherwise \eqref{eq:pmocc} and \eqref{eq:pmpocc}
    would be contained in the $p^+$-occurrence
    \begin{equation*}
        \Big( u_{i 0} \prod_{j = 1}^{J - 1} v_{i j}^{m_j} u_{i j} \cdot p_1,\
        p^M y p^{M'},\
        p_2' u_{i J'} \prod_{j = J' + 1}^{r_i} v_{i j}^{m_j} u_{i j} \Big),
    \end{equation*}
    and then the maximal occurrences $(w_1, p^N, w_2)$ and $(w_1', p^{N'}, w_2')$ would be the same
    by Lemma~\ref{lem:overlap-maxocc}.

    Finally, we prove Condition~\ref{cond:3}.
    Let $u \in L$.
    There are numbers $i, m_1, \dots, m_{r_i}$ such that
    \begin{equation*}
        u = u_{i 0} \prod_{j = 1}^{r_i} v_{i j}^{m_j} u_{i j}.
    \end{equation*}
    Let $x$ be a factor of $u$ of length at least $k$.
    If it does not have a common factor of length at least $|v_{i j}^{n + 2}|$
    with the factor $v_{i j}^{m_j}$
    for any $j$, then
    \begin{equation*}
        |x| < \sum_{j = 0}^{r_i} |u_{i j}| + \sum_{j = 1}^{r_i} |v_{i j}^{n + 2}| \leq k,
    \end{equation*}
    which is a contradiction.
    So there exists a number $j$ such that $x$ and $v_{i j}^{m_j}$
    have a common factor of length at least $|v_{i j}^{n + 2}|$,
    and this common factor necessarily has a $p^{n + 1}$-occurrence,
    where $p$ is the Lyndon root of $v_{i j}$.
\end{proof}

Now we are ready to prove our main theorem.

\begin{theorem} \label{thm:main}
    A regular language $L$ has a finite SSF
    if and only if
    $L$ does not have a subset of the form $x w^* y w^* z$
    for any words $w, x, y, z$ such that $wy \ne yw$.
\end{theorem}

\begin{proof}
    The ``only if'' direction follows from Lemma~\ref{lem:half-main}.
    To prove the ``if'' direction,
    let $n, k, P$ be as in Lemma~\ref{lem:nkp}
    ($L$ is bounded by Lemma~\ref{lem:unbounded}).
    Let $u, v \in L$ be $k$-abelian equivalent.
    We are going to show that $u = v$.
    This proves the theorem by Lemma~\ref{lem:kabel}.
    If $|u| = |v| < k$, then trivially $u = v$, so we assume that $|u| = |v| \geq k$.

    Let $P_j = \{p^i \mid p \in P, i \geq j\}$ for all $j$.
    Let the maximal $P_n$-occurrences in $u$ be
    \begin{equation} \label{eq:umaxocc}
        (x_1, p_1^{m_1}, x'_1), \dots, (x_r, p_r^{m_r}, x'_r),
    \end{equation}
    where $p_1, \dots, p_r \in P$.
    It follows from $|u| \geq k$ and Condition~\ref{cond:3} of Lemma~\ref{lem:nkp} that $r \geq 1$.
    We can assume that the occurrences have been ordered so that $|x_1| \leq \dots \leq |x_r|$.
    By Condition~\ref{cond:2} of Lemma~\ref{lem:nkp},
    the words $p_1, \dots, p_r$ are pairwise distinct.
    All $P_n$-occurrences in $u$
    are contained in one of the maximal occurrences~\eqref{eq:umaxocc}.
    By Condition~\ref{cond:1} of Lemma~\ref{lem:nkp},
    $p^n$ cannot be a factor of $p_j^{m_j}$ if $p \in P \smallsetminus \{p_j\}$,
    so if $p \in P \smallsetminus \{p_1, \dots, p_r\}$,
    then there are no $p^{\geq n}$-occurrences in $u$,
    and all $p_i^{\geq n}$-occurrences are $(x_i p_i^l, p_i^j, p_i^{m_i - j - l} x_i')$
    for $j \in \{n, \dots, m_i\}$ and $l \in \{0, \dots, m_i - j\}$.
    In particular, $|u|_{p_i^n} = m_i - n + 1$.

    Similarly, let the maximal $P_n$-occurrences in $v$ be
    \begin{equation*}
        (y_1, q_1^{n_1}, y'_1), \dots, (y_s, q_s^{n_s}, y'_s),
    \end{equation*}
    where $s \geq 1$ and $q_1, \dots, q_s \in P$.
    As above,
    we can assume that the occurrences have been ordered so that $|y_1| \leq \dots \leq |y_s|$,
    and we can prove that
    the words $q_1, \dots, q_s$ are pairwise distinct,
    $p^n$ cannot be a factor of $q_j^{n_j}$ if $p \in P \smallsetminus \{q_j\}$,
    and if $p \in P \smallsetminus \{q_1, \dots, q_s\}$,
    then there are no $p^{\geq n}$-occurrences in $v$,
    all $q_i^{\geq n}$-occurrences are $(y_i q_i^l, q_i^j, q_i^{n_i - j - l} y_i')$
    for $j \in \{n, \dots, n_i\}$ and $l \in \{0, \dots, n_i - j\}$,
    and $|v|_{q_i^n} = n_i - n + 1$.

    If $p \in P$, then $|p^n| < k$ by Condition~\ref{cond:3} of Lemma~\ref{lem:nkp},
    and then $|u|_{p^n} = |v|_{p^n}$ because $u \kabel{k} v$.
    It follows that $r = s$ and $\{p_1, \dots, p_r\} = \{q_1, \dots, q_s\}$.
    We have seen that $|u|_{p_i^n} = m_i - n + 1$ and $|v|_{q_j^n} = n_j - n + 1$,
    so if $p_i = q_j$, then $m_i = n_j$.

    We prove by induction that $(x_i, p_i, m_i) = (y_i, q_i, n_i)$ for all $i \in \{1, \dots, r\}$.
    First, we prove the case $i = 1$.
    The words $u$ and $v$ have prefixes $x_1 p_1^n$ and $y_1 q_1^n$, respectively.
    There is only one $P_n$-occurrence
    and no $P_{n + 1}$-occurrences in $x_1 p_1^n$.
    Similarly, there is only one $P_n$-occurrence
    and no $P_{n + 1}$-occurrences in $y_1 q_1^n$.
    By Condition~\ref{cond:3} of Lemma~\ref{lem:nkp},
    $|x_1 p_1^n| < k$ and $|y_1 q_1^n| < k$.
    Because $u$ and $v$ are $k$-abelian equivalent, they have the same prefix of length $k - 1$,
    and thus one of $x_1 p_1^n$ and $y_1 q_1^n$ is a prefix of the other.
    If, say, $x_1 p_1^n$ is a prefix of $y_1 q_1^n$,
    then $y_1 q_1^n$ has an occurrence $(x_1, p_1^n, z)$ for some word $z$,
    and this must be the unique $P_n$-occurrence $(y_1, q_1^n, \eps)$.
    It follows that $x_1 = y_1$ and $p_1 = q_1$,
    and then also $m_1 = n_1$.

    Next, we assume that $(x_i, p_i, m_i) = (y_i, q_i, n_i)$ for some $i \in \{1, \dots, r - 1\}$
    and prove that $(x_{i + 1}, p_{i + 1}, m_{i + 1}) = (y_{i + 1}, q_{i + 1}, n_{i + 1})$.
    Let $x_{i + 1} = x_i p_i^{m_i - n} x_i''$
    and $y_{i + 1} = y_i q_i^{n_i - n} y_i'' = x_i p_i^{m_i - n} y_i''$.
    The unique shortest factor in $u$ beginning with $p_i^n$
    and ending with $p^n$ for some $p \in P \smallsetminus \{p_i\}$
    is the factor $x_i'' p_{i + 1}^n$
    starting at position $|x_i p_i^{m_i - n}|$
    and ending at position $|x_{i + 1} p_{i + 1}^n|$.
    Similarly, the unique shortest factor in $v$ beginning with $p_i^n$
    and ending with $p^n$ for some $p \in P \smallsetminus \{p_i\}$
    is the factor $y_i'' q_{i + 1}^n$
    starting at position $|y_i q_i^{n_i - n}| = |x_i p_i^{m_i - n}|$
    and ending at position $|y_{i + 1} q_{i + 1}^n|$.
    There are no $P_{n + 1}$-occurrences in these factors,
    so they are of length less than $k$ by Condition~\ref{cond:3} of Lemma~\ref{lem:nkp},
    and they must be equal because $u \kabel{k} v$.
    It follows that $p_{i + 1} = q_{i + 1}$, $x_i'' = y_i''$, and $x_{i + 1} = y_{i + 1}$,
    and then also $m_{i + 1} = n_{i + 1}$.

    It follows by induction that $x_r p_r^{m_r} = y_r q_r^{n_r}$.
    Because $|u| = |v|$, it must be $|x_r'| = |y_r'|$.
    Because $x_r'$ does not have any $P_{n + 1}$-occurrences,
    $|x_r'| < k$ by Condition~\ref{cond:3} of Lemma~\ref{lem:nkp}.
    Because $u$ and $v$ are $k$-abelian equivalent, they have the same suffix of length $k - 1$,
    so $x_r' = y_r'$.
    Thus $u = v$.
    This completes the proof.
\end{proof}

\begin{example}
    First, consider the language $K = a^* (abab)^* ba (ba)^*$.
    It has a subset
    \begin{math}
        (abab)^* ba (ba)^* = (abab)^* b (ab)^* a,
    \end{math}
    which has a subset
    \begin{math}
        (abab)^* b (abab)^* a.
    \end{math}
    It follows from Theorem~\ref{thm:main} that $K$ does not have a finite SSF.

    Then, consider the language $L = a^* (abab)^* aba (ba)^*$.
    We can write
    \begin{equation*}
        L = a^* (abab)^* (ab)^* aba = a^* (ab)^* aba.
    \end{equation*}
    It can be proved that if $L$ has a subset $x w^* y w^* z$ with $w \ne \eps$,
    then the Lyndon root of $w$ is $a$ or $ab$, and $wy = yw$.
    It follows from Theorem~\ref{thm:main} that $L$ has a finite SSF.
\end{example}

\section{Conclusion}

In this article, we have defined and studied separating sets of factors.
In particular, we have considered the question of whether a given language has a finite SSF.
We have answered this question for sets of factors of infinite words and for regular languages.
In the future, this question could be studied for other families of languages.
We can also ask the following questions:
\begin{itemize}
    \item
    Given a language with a finite SSF,
    what is the minimal size of an SSF of this language?
    For example, this question could be considered for $\Sigma^n$.
    \item
    Given a language with no finite SSF,
    how ``small'' can the growth function of an SSF of this language be?
    For example, this question could be considered for $\Sigma^*$.
\end{itemize}

\bibliography{ref}

\end{document}